\documentclass[12pt]{article}
\usepackage{amsmath, amscd, amssymb, latexsym, epsfig, color, amsthm}
\numberwithin{equation}{section}
\def\proof{\smallskip\noindent {\it Proof: \ }} \def\endproof{\hfill$\square$\medskip}
\newtheorem{theorem}{Theorem}[section]
\newtheorem{proposition}[theorem]{Proposition}

\newtheorem{corollary}[theorem]{Corollary}
\newtheorem{conjecture}[theorem]{Conjecture}
\newtheorem{lemma}[theorem]{Lemma}

\theoremstyle{definition}
\newtheorem{defn}[theorem]{Definition}
\newtheorem{example}[theorem]{Example}
\newtheorem{remark}[theorem]{Remark}

\def\R{\mathbb{R}}
\def\Z{\mathbb{Z}}

\def\M{\mathcal{M}}
\def\D{\mathcal{D}}
\def\Sw{\mathcal{S}}
\def\C{\mathcal{C}}
\def\B{\mathcal{B}}
\def\S{\mathbb{S}}

\DeclareMathOperator{\SwEl}{\mathcal{S}el}
\DeclareMathOperator{\lex}{lex}
\DeclareMathOperator{\Fill}{fill}
\DeclareMathOperator\lk{\mathrm{lk}}
\DeclareMathOperator\Star{\mathrm{st}}

\def\Res{\mathcal{R}}
\begin{document}
\title{\vspace{-1ex}Centrally symmetric manifolds with few vertices}
\author{
Steven Klee \thanks{Research partially supported by NSF VIGRE grant DMS-0636297}\\
\small Department of Mathematics\\[-0.8ex]
\small One Shields Ave.\\[-0.8ex]
\small University of California, Davis, CA 95616, USA\\[-0.8ex]
\small \texttt{klee@math.ucdavis.edu}
\and Isabella Novik
\thanks{Research partially supported by NSF grant DMS-0801152}\\
\small Department of Mathematics, Box 354350\\[-0.8ex]
\small University of Washington, Seattle, WA 98195-4350, USA\\[-0.8ex]
\small \texttt{novik@math.washington.edu}
}
\maketitle
\begin{abstract}
A centrally symmetric $2d$-vertex combinatorial triangulation of the product of spheres $\S^i\times\S^{d-2-i}$ is constructed for all pairs of non-negative integers $i$ and $d$ with $0\leq i \leq d-2$. For the case of $i=d-2-i$, the existence of such a triangulation was conjectured by Sparla. The constructed complex admits a vertex-transitive action by a group of order $4d$. The crux of this construction is a definition of a certain full-dimensional subcomplex, $\B(i,d)$, of the boundary complex of the $d$-dimensional cross-polytope. This complex $\B(i,d)$ is a combinatorial manifold with boundary and its boundary provides a required triangulation of $\S^i\times\S^{d-i-2}$. Enumerative characteristics of $\B(i,d)$ and its boundary, and connections to another conjecture of Sparla are also discussed.
\end{abstract}
\section{Introduction}
What is the minimum number of vertices needed to triangulate a given (triangulable) manifold? How will the answer change if we require a triangulation to be centrally symmetric (i.e., possess a free involution)? Starting from the seminal work of Ringel and Youngs \cite{RingYoung, Ringel}, and Walkup \cite {Walkup}, this question has motivated a tremendous amount of research in topological combinatorics and combinatorial topology, see for instance K\"uhnel's book \cite{Kuhnel-book}, a forthcoming  book by Lutz \cite{Lutz-book} parts of which are available electronically at \cite{Lutz-web}, and many references mentioned there.

Of a particular interest are centrally symmetric (cs, for short) triangulations of products of spheres. It is well-known and easy to see that an arbitrary cs triangulation $\Delta$ of $\S^i\times \S^{d-i-2}$ has at least $2d$ vertices. (Indeed, such a triangulation necessarily contains two vertex-disjoint $(d-2)$-simplices, and hence has at least $2(d-1)$ vertices. Moreover, if $\Delta$ had only $2(d-1)$ vertices, it would be a full-dimensional subcomplex of the boundary complex of the $(d-1)$-dimensional cross polytope, which is a combinatorial $(d-2)$-dimensional sphere. This is however impossible as no closed manifold but a sphere is embeddable in a sphere of the same dimension.) The natural question is then whether there exist cs triangulations of $\S^i\times \S^{d-i-2}$ with exactly $2d$ vertices.  Our main theorem is a {\em positive} answer to this question.

The first result in this series is due to K\"uhnel and Lassmann \cite{KuhLass} who constructed a cs $2d$-vertex triangulation of $\S^1\times\S^{d-2}$ for all $d\geq 2$. This appears to be the only infinite family of cs triangulations of products of spheres (with $2d$ vertices) that was known until now.

In his Doctoral thesis \cite{Sparla-thesis}, Sparla constructed a cs 12-vertex triangulation of $\S^2 \times \S^2$, see also \cite{LassSp}, and conjectured that there exists a cs $4k$-vertex triangulation of $\S^{k-1} \times \S^{k-1}$ for every $k$. Lutz \cite{Lutz-thesis}, with an aid of computer programs MANIFOLD\_VT and BISTELLAR, confirmed this conjecture for $k=4$ and $k=5$ as well as found many cs $2d$-vertex triangulations of $\S^i\times \S^{d-i-2}$ for $d\leq 10$. Very recently, Effenberger \cite{Eff-thesis} proposed a certain construction of cs simplicial complexes with $4k$ vertices
that conjecturally triangulate $\S^{k-1}\times \S^{k-1}$; with the help of the software package {\em simcomp} he then verified that this indeed holds for all values of $k\leq 12$, thus establishing Sparla's conjecture up to $k=12$.

Our main result provides a cs $2d$-vertex triangulation of $\S^i\times \S^{d-i-2}$ for {\em all} nonnegative integers $0\leq i\leq d-2$, and in particular settles Sparla's conjecture in full generality. In the following, we denote by $\D_m$ the dihedral group of order $2m$.

\begin{theorem} \label{main-S}
For all pairs of integers $(i,d)$ with $0\leq i\leq d-2$, there exists a centrally symmetric $2d$-vertex triangulation of $\S^i\times \S^{d-i-2}$. This triangulation admits a vertex-transitive action by the dihedral group of order $4d$, $\D_{2d}$, if at least one of the numbers $i$ and $d-i$ is odd, and by the group $\Z_2\times \D_d$ otherwise. \end{theorem}

The last part of Theorem \ref{main-S} proves Conjecture 4.9 from \cite{Lutz-thesis} for all $d\not\equiv 2 \mod 4$. This conjecture asserts existence of cs $2d$-vertex triangulations of $\S^{\lfloor\frac{d}{2}\rfloor-1}\times \S^{\lceil\frac{d}{2}\rceil-1}$ admitting a vertex-transitive dihedral group action. Further,  Lutz \cite{Lutz-thesis} has shown that no cs triangulation of $\S^2 \times \S^4$ on $16$ vertices admits a vertex-transitive action by a cyclic group of order $16$, and no cs triangulation of $\S^2 \times \S^6$ on $20$ vertices admits a vertex-transitive action by a dihedral group of order $40$.  As such, the parity distinction in Theorem \ref{main-S} cannot be avoided.

The crux of the proof of Theorem \ref{main-S} is a construction of a certain simplicial complex, $\B(i,d)$ (for all $0\leq i\leq d-1$) that is rather easy to analyze. This complex is constructed as a pure full-dimensional subcomplex of the boundary complex of the $d$-dimensional cross polytope. (In fact, for $i=d-1$, $\B(i,d)$ is the entire boundary complex of the cross polytope.) Theorem \ref{main-S} follows once we establish the following properties of $\B(i,d)$.

\begin{theorem} \label{main-B}
For $0\leq i <d-1$, the complex $\B(i,d)$ satisfies the following:
\begin{enumerate}
\item[(a)] $\B(i,d)$ contains the entire $i$-skeleton of the $d$-dimensional cross polytope as a subcomplex.
\item[(b)] $\B(i,d)$ is centrally symmetric. Moreover, it admits a vertex-transitive action of $\Z_2\times\D_d$ if $i$ is even and of $\D_{2d}$ if $i$ is odd.
\item[(c)] The complement of $\B(i,d)$ in the boundary complex of the $d$-dimensional cross polytope (that is, the complex generated by the facets of the cross polytope that are not in $\B(i,d)$) is simplicially isomorphic to $\B(d-i-2,d)$.
\item[(d)] $\B(i,d)$ is a combinatorial manifold (with boundary) whose
integral (co)homology groups coincide with those of $\S^i$.
\item[(e)] The boundary of $\B(i,d)$ is homeomorphic to $\S^i\times \S^{d-i-2}$.
\end{enumerate}
\end{theorem}

The construction of $\B(i,d)$ is so simple to state that we cannot resist a temptation to sketch it right now. More details will be given in Section 3. Let $C^*_d$ denote the boundary complex of the $d$-dimensional cross polytope on the vertex set $\{x_1,\ldots,x_d,y_1,\ldots,y_d\}$, where the labeling is such that for every $j$, $x_j$ and $y_j$ are antipodal vertices of $C^*_d$. Then each facet $\tau$ of $C^*_d$ can be identified with a word, $w(\tau)$, of length $d$ in the alphabet $\{x,y\}$: the $i$-th entry of $w(\tau)$ is $x$ if $x_i\in \tau$ and it is $y$ otherwise.  For instance, $xxyyy$ encodes the facet $\{x_1, x_2, y_3,y_4, y_5\}$ of $C^*_5$.  For each word, $u=u_1...u_d$ of length $d$ in the $\{x,y\}$-alphabet count the number of indices $1\leq j\leq d-1$ such that $u_j\neq u_{j+1}$, that is, count the number of switches from $x$ to $y$ and $y$ to $x$. For example, in $xyxxyyy$ there are 3 such switches occurring at positions $j=1,2,4$. We {\bf define} $\B(i,d)$ to be a pure subcomplex of $C^*_d$ generated by all the facets encoded by words {\bf with at most $i$ switches}. Thus $\B(0,d)$ is generated by the two facets of $C^*_d$ with zero switches, namely $\{x_1, x_2,\ldots,x_d\}$ and $\{y_1,y_2,\ldots, y_d\}$, and so it is a disjoint union of two $(d-1)$-simplices. On the other hand, for $i=d-1$ as many switches as possible are allowed, and hence $\B(d-1,d)$ is the entire $C^*_d$. The boundary of the complex $\B(1,4)$ is pictured in Figure \ref{torus}; note that $\B(1,4)$ and its complement in $C^*_4$ provide the classical decomposition of $\S^3$ as the union of two solid tori $\S^1\times\mathbb{B}^2$ glued together along their common boundary.

\begin{figure}[ht]\label{torus}
\begin{center}
\epsfig{file = 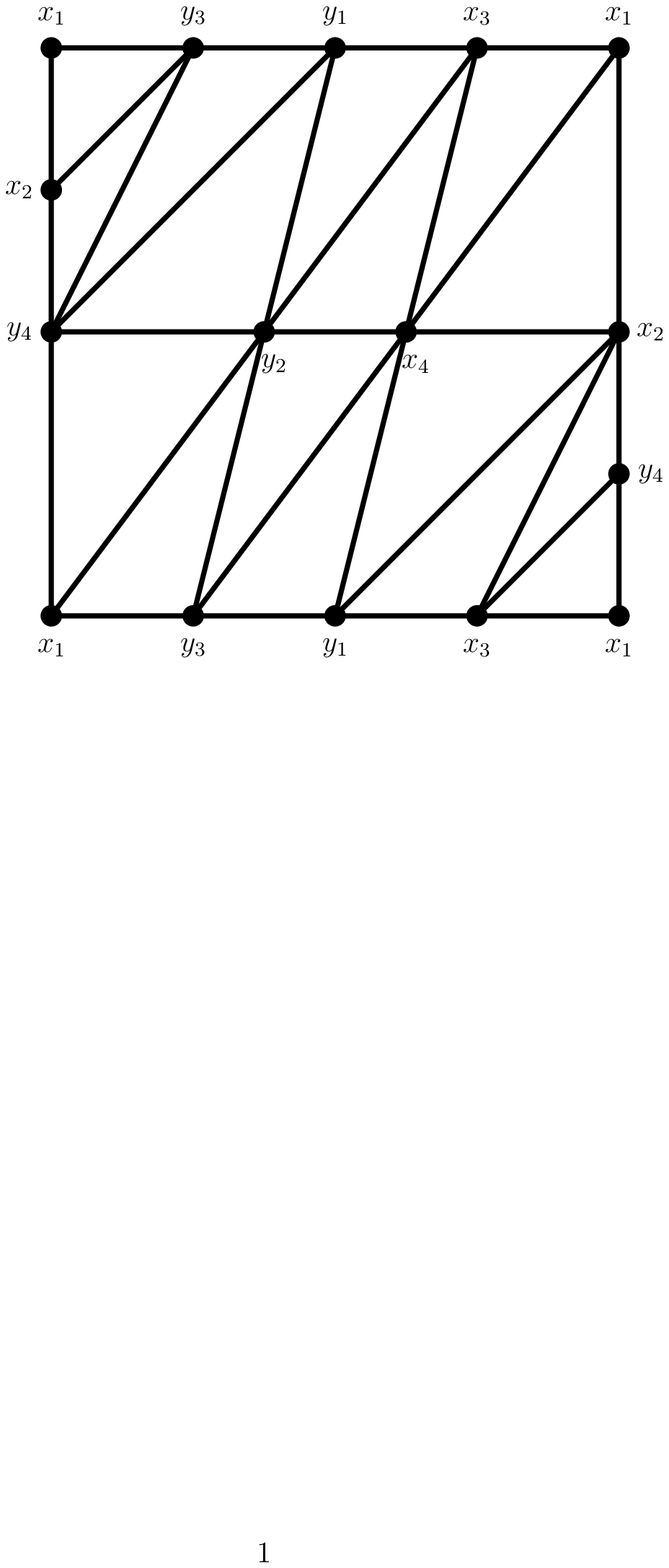, height = 5cm, bbllx = 200, bblly=450,
bburx=470, bbury=710, clip=}
\end{center}
\caption{$\partial\B(1,4)$}
\end{figure}

The rest of the paper is structured as follows. In Section 2 we review basic facts related to simplicial complexes and combinatorial manifolds. Section 3 is a purely combinatorial section devoted to the proof of parts (a)--(d) of Theorem \ref{main-B}. Section 4 is more topological and contains the proof of part (e) along with derivation of Theorem \ref{main-S} from Theorem \ref{main-B}. We close in Section 5 with several results pertaining to face enumeration and connections to another conjecture by Sparla.

\section{Preliminaries}
Here we briefly review several notions and results
related to simplicial complexes and combinatorial manifolds
as well as set up some notation.

A {\em simplicial complex} $\Delta$ on the vertex set $V$ is a collection of subsets of $V$ that is closed under inclusion
and contains all singletons $\{v\}$ for $v\in V$.
The elements of $\Delta$ are called its {\em faces}.
For $\sigma\in\Delta$, set $\dim \sigma:=|\sigma|-1$ and define the {\em dimension} of $\Delta$, $\dim \Delta$, as the maximal dimension of its faces. The {\em $i$-skeleton} of $\Delta$ is the collection of
all faces of $\Delta$ of dimension at most $i$.
The {\em facets} of $\Delta$
are maximal (under inclusion) faces of $\Delta$. We say that $\Delta$ is {\em pure} if all of its facets have
the same dimension.

Let $\Delta$ be a pure $(d-1)$-dimensional
simplicial complex. For $\sigma\in\Delta$, denote
by $2^{\sigma}$ the simplex $\sigma$ together with all of its faces.
A {\em shelling} of $\Delta$ is an ordering $(\tau_1, \tau_2, \ldots, \tau_s)$ of its facets such that for all $1<i\leq s$, the complex
$2^{\tau_i}\cap(\cup_{j<i}2^{\tau_j})$
is pure of dimension $d-2$. Equivalently, $(\tau_1, \tau_2, \ldots, \tau_s)$ is
a shelling if for every $1\leq i\leq s$, the collection of faces $2^{\tau_i}-(\cup _{j<i}2^{\tau_j})$
has a unique minimal element (with respect to inclusion); this minimal face is called the {\em restriction} of $\tau_i$ and is denoted $\Res(\tau_i)$.

If $\Delta$ is a simplicial complex and $\sigma$ is a face of $\Delta$, then the {\em link} of $\sigma$ in $\Delta$, $\lk_\Delta \sigma$, and the {\em star} of $\sigma$ in $\Delta$, $\Star_\Delta \sigma$, are defined by
\begin{displaymath}
\lk_\Delta \sigma=\lk \sigma:= \{\tau-\sigma\in\Delta \ : \ \sigma\subseteq \tau \in \Delta \}\, \mbox{ and } \, \Star_\Delta \sigma=\Star \sigma:=\{\tau\in\Delta \ : \ \sigma\cup \tau\in\Delta\}.
\end{displaymath}

A $(d-1)$-dimensional simplicial complex $\Delta$ is called a {\em combinatorial manifold} if the link of every non-empty face $\sigma$ of $\Delta$ is a triangulated $(d-|\sigma|-1)$-dimensional ball or sphere. A combinatorial ball (sphere) is a combinatorial manifold that triangulates a ball (sphere).

A well-known result due to Danaraj and Klee \cite{DanarajKlee} asserts that if a $(d-1)$-dimensional simplicial complex $\Delta$ is shellable and if, in addition, each $(d-2)$-dimensional face of $\Delta$ is contained in no more than two facets, then $\Delta$ is a combinatorial ball or combinatorial sphere. Therefore, a proper, full-dimensional, shellable subcomplex of the boundary complex of a simplicial polytope is a combinatorial ball.

All simplicial complexes considered in this paper are subcomplexes of the boundary complex of a cross polytope. Consider $d$ affinely independent points in $\R^d-\{0\}$, say, $x_1,\ldots, x_d$, and let $y_i=-x_i\in\R^d$ for $1\leq i \leq d$. A $d$-dimensional {\em cross polytope} is the convex hull of the set $\{x_1,\ldots, x_d, y_1,\ldots, y_d\}$. All $d$-dimensional cross polytopes are affinely equivalent simplicial polytopes. The boundary complex of the $d$-dimensional cross polytope, denoted $C^*_d$, is thus a pure simplicial complex on the vertex set $V_d=V:=\{x_1,\ldots, x_d, y_1,\ldots, y_d\}$ (that we fix from now on) whose facets are the subsets of $V_d$ containing exactly one element from $\{x_j,y_j\}$ for each $1\leq j \leq d$. Hence (i) $C^*_{d-1}$ is a subcomplex of $C^*_d$ induced by $V_{d-1}\subset V_d$, and (ii) the set of facets of $C^*_d$ is in natural bijection with the set of $xy$-words of length $d$: a facet $\tau\in C^*_d$ is encoded by a word $w(\tau)=u_1...u_d$, where $u_i=x$ if $x_i\in C^*_d$ and $u_i=y$ otherwise; conversely, an $xy$-word $u=u_1...u_d$ encodes a facet $F(u)=\{(u_1)_1,\ldots, (u_d)_d\}$. For example, the facet of $C^*_5$ encoded by $u=xyxxy$ is $F(u)=\{x_1,y_2,x_3,x_4,y_5\}$.

We will also need a few standard facts from homology theory, such as the Mayer-Vietoris sequence (see Hatcher's book \cite{Hatcher} for reference). Throughout the paper, we denote by $H_j(\Delta; \Z)$ ($\widetilde{H}_j(\Delta; \Z)$, resp.) the $j$-th simplicial homology (reduced simplicial homology, resp.) of $\Delta$ computed with coefficients in $\Z$.
\section{The main construction}

In this section we present our main construction --- the family of complexes $\B(i,d)$, and study various combinatorial properties that these complexes possess.

Write $[d-1]$ for the set $\{1,2,\ldots, d-1\}$. For an $xy$-word $u=u_1 \ldots u_d$ of length $d$, define the {\em switch set} of $u$, $\Sw_d(u)=\Sw(u):=\{j\in[d-1]\ : u_j\neq u_{j+1}\}$. Using the above identification between the facets of $C^*_d$ and $xy$-words of length $d$, define the {\em switch set} of a facet $\tau\in C^*_d$ by $\Sw_d(\tau)=\Sw(\tau):=\Sw(w(\tau))$. (When working with a fixed $d$, we will omit the subscripts.)

\begin{defn}
For $-1\leq i \leq d-1$, the complex $\B(i,d)$ is a pure full-dimensional subcomplex of $C^*_d$ whose facets are the facets of $C^*_d$ with switch set of size at most $i$. \end{defn}

Thus, $\B(i-1,d)\subset\B(i,d)$ for all $0\leq i\leq d-1$; $\B(-1,d)$ is the empty complex and $\B(d-1,d)=C^*_d$; $\B(0,d)=2^{\{x_1,x_2,\ldots, x_d\}} \cup 2^{\{y_1,y_2,\ldots, y_d\}}$ is a disjoint union of two simplices, and $\B(d-2,d)$ is $C^*_d$ with two facets (the ones identified with $xyxy\ldots$ and $yxyx\ldots$) removed; $\B(1,d)$ has $2d$ facets: they are the two facets of $\B(0,d)$ together with facets of the form $\{x_1, x_2,\ldots, x_{j}, y_{j+1},\ldots, y_d\}$ and $\{y_1,y_2,\ldots, y_j, x_{j+1},\ldots, x_d\}$ for $1\leq j \leq d-1$.  The complex $\B(1,3)$ is shown in Figure 2. 

\begin{figure}[ht]\label{annulus}
\begin{center}
\epsfig{file = 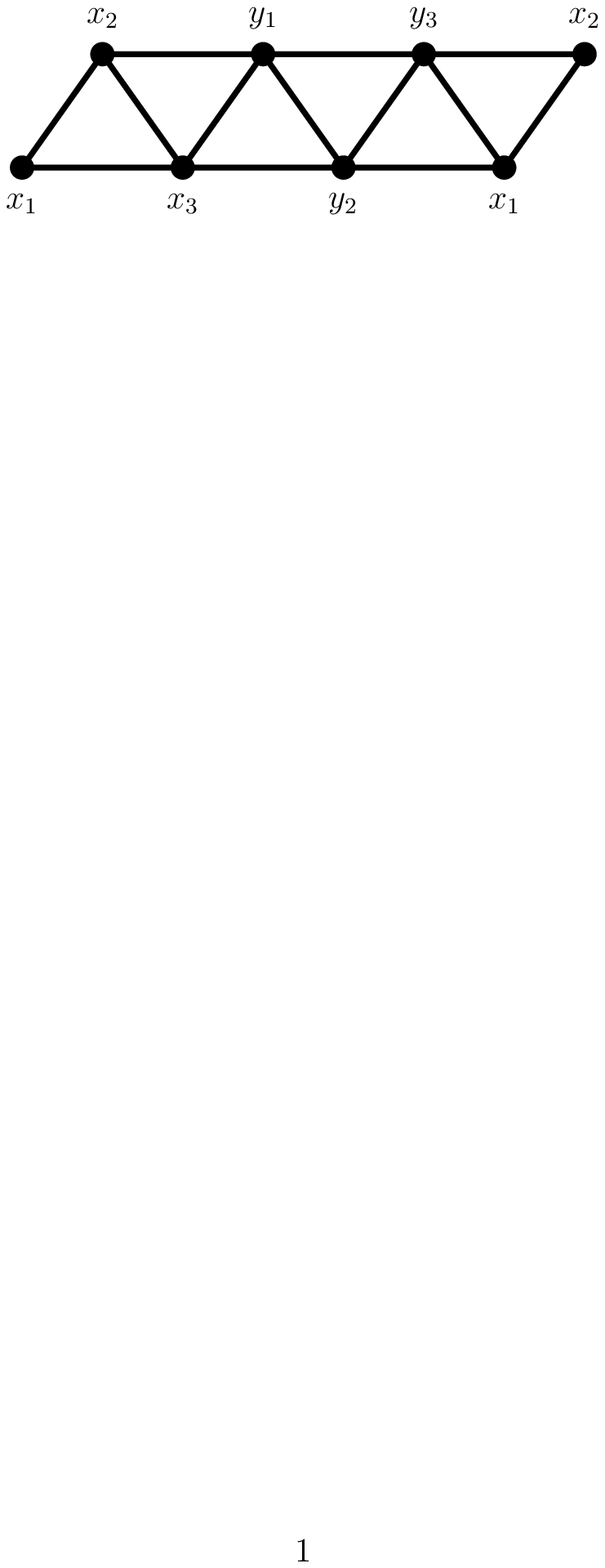, height = 2cm, bbllx = 190, bblly=560, bburx=415, bbury=640, clip=}
\end{center}
\caption{$\B(1,3)$}
\end{figure}

What are the smaller-dimensional faces of $\B(i,d)$?  If $\sigma$ is any face of $C^*_d$, then $\sigma$ is of the form $\{z_{j_1},z_{j_2},\ldots,z_{j_s}\}$ for some $1\leq j_1<\ldots<j_s\leq d$ and $z_{j_k}\in\{x_{j_k},y_{j_k}\}$ for all $1\leq k \leq s$. Set $j_0=0$. For $1\leq k \leq s$ and for $j_{k-1}<j< j_k$, define \begin{displaymath}
z_j:=\left\{\begin{array}{ll}
x_j & \mbox{ if $z_{j_k}=x_{j_k}$,}\\
y_j & \mbox{ otherwise}.
\end{array}
\right.
\end{displaymath}
Also for all $j_{s}<j\leq d$, define $z_j:=x_j$ if $z_{j_{s}}=x_{j_{s}}$ and define $z_j:=y_j$ otherwise. We call the facet $\tau:=\{z_1,\ldots, z_d\}$ of $C^*_d$, the {\em filling} of $\sigma$ in $C^*_d$, and write $\tau=\Fill_d(\sigma)$. Observe that $\sigma\subseteq \Fill_d(\sigma)$ and that if $\tau'$ is any other facet of $C^*_d$ containing $\sigma$, then the size of the switch set of $\tau'$ is at least as large as that of the switch set of $\Fill_d(\sigma)$. This establishes the following lemma.
\begin{lemma} \label{small-faces}
A face $\sigma$ of $C^*_d$ is a face of $\B(i,d)$ if and only $|\Sw_d(\Fill_d(\sigma))|\leq i$.
\end{lemma}

We are now in a position to verify parts (a)--(c) of Theorem \ref{main-B}. All of them follow easily from our definition of $\B(i,d)$.

\medskip\noindent{\it Proof of Theorem \ref{main-B}(a):\,}
To show that $\B(i,d)$ contains the entire $i$-skeleton of $C^*_d$, consider an $i$-face $\sigma$ of $C^*_d$. Then $\sigma=\{z_{j_1}, z_{j_2}, \ldots, z_{j_{i+1}}\}$ for some $1\leq j_1 < \cdots < j_{i+1} \leq d$ and $z_{j_k}\in\{x_{j_k}, y_{j_k}\}$ for $k \in [i+1]$. It follows from the definition of filling that $\Sw(\Fill_d(\sigma)) \subseteq \{j_1,\ldots, j_i\}$,
and hence has size at most $i$. Thus by Lemma \ref{small-faces}, $\sigma\in\B(i,d)$, and the result follows. \endproof

\medskip\noindent{\it Proof of Theorem \ref{main-B}(b):\,}

We first treat the case that $i$ is even. To show that $\B(i,d)$ is centrally symmetric and, in fact, admits a vertex-transitive action by $\Z_2\times\D_d$, define three permutations, $D$, $E$, and $R$, on the vertex set $V_d$ of $C^*_d$ as follows: \begin{itemize}
\item $D$ maps $x_j$ to $y_j$, and $y_j$ to $x_j$; this permutation has order 2.
\item $E$ maps $x_j$ to $x_{d-j+1}$, and $y_j$ to $y_{d-j+1}$; this permutation has order 2.
\item $R$ maps $x_j$ to $x_{j+1}$ and $y_j$ to $y_{j+1}$, where the addition is modulo $d$ (so $R(x_d)=x_1$); this permutation has order $d$.
\end{itemize}
All three of these maps induce a simplicial automorphism of $C^*_d$. In particular, each of these maps defines a permutation on the set of facets of $C^*_d$.  By using our identification between the facets of $C^*_d$ and $xy$-words of length $d$, each of these maps also acts as a permutation on the set of words: for an $xy$-word $u=u_1\ldots u_d$, $D$ replaces each letter in $u$ by its opposite (i.e., $x$ by $y$ and $y$ by $x$), $E$ reverses the order of letters in $u$, and $R$ takes the last letter of $u$ and moves it to the front. Thus for any facet $\tau$ of $C^*_d$, $\Sw(D(\tau))=\Sw(\tau)$ and $|\Sw(E(\tau))|=|\Sw(\tau)|$, yielding that $D$ and $E$ are involutions on $\B(i,d)$.
Also, the above description of $R$ implies that $|\Sw(R(\tau))|\leq |\Sw(\tau)|+1$, and so if $|\Sw(\tau)|\leq i-1$, then $|\Sw(R(\tau))|\leq i$.  On the other hand, if $|\Sw(\tau)|=i$, then since $i$ is even, the first and the last letters of $w(\tau)$ --- the $xy$-word corresponding to $\tau$ --- are the same, and hence moving the last letter of $w(\tau)$ to the front does not increase the size of the switch set. We infer that if $|\Sw(\tau)|\leq i$, then $|\Sw(R(\tau))|\leq i$, and so $R$ acts as a permutation on the facets of $\B(i,d)$. As $ERE=R^{-1}$ and $D$ commutes with both $E$ and $R$, it follows that $D$, $E$, and $R$ generate the group $\Z_2\times \D_d$ (in the group of all permutations of $2d$ vertices) that acts transitively on $V$, yielding the result.

The case of an odd $i$ is almost identical, just replace $R$ in the above proof with the map $R'$ that sends $x_d$ to $y_1$, $y_d$ to $x_1$, and is defined by $R'(x_j)=x_{j+1}$ and $R'(y_j)=y_{j+1}$ for $j\in [d-1]$. Then for a facet $\tau$, $|\Sw(R'(\tau))|\leq |\Sw(\tau)|+1$. Moreover, if $|\Sw(\tau)|=i$, then $|\Sw(R'(\tau))|\leq |\Sw(\tau)|$: this is because for $i$ odd, any $xy$-word $u_1\ldots u_d$ with exactly $i$ switches has opposite first and last letters: $u_1\neq u_d$. The result follows since $E$ and $R'$ generate the dihedral group of order $4d$.
\endproof

\medskip\noindent{\it Proof of Theorem \ref{main-B}(c):\,}
Let $A: V \to V$ be an involution on $V$ defined by $x_j\mapsto x_j$ and $y_j \mapsto y_j$ for $j$ odd, and by $x_j\mapsto y_j$ and $y_j \mapsto x_j$ for $j$ even. Then for any facet $\tau\in C^*_d$, $\Sw(A(\tau))=[d-1]-\Sw(\tau)$. Thus $|\Sw(\tau)|\leq d-i-2$ if and only if $|\Sw(A(\tau))|\geq i+1$, and hence $A$ is a simplicial isomorphism between $\B(d-i-2,d)$ and the complement of $\B(i,d)$. \endproof

The proof of Theorem \ref{main-B}(d) takes a bit more work and requires the following lemmas.
\begin{lemma} \label{links-intersection}
The intersection of the links of $x_d$ and $y_d$ in $\B(i,d)$ is $\B(i-1,d-1)$.
\end{lemma}
\begin{lemma} \label{shellability}
The stars of $x_d$ and $y_d$ in $\B(i,d)$ are shellable $(d-1)$-dimensional complexes.
\end{lemma}

Assuming the lemmas, the {\it proof of Theorem \ref{main-B}(d)} is almost immediate. We use induction on $i$. For $i=0$, $\B(0,d)$ is a disjoint union of two $(d-1)$-dimensional simplices, and so it is a combinatorial manifold that retracts onto $\S^0$.  For $0<i<d-1$, we proceed as follows. Since every facet of $C^*_d$, and hence also of $\B(i,d)$, contains either $x_d$ or $y_d$, it follows that \begin{equation} \label{union}
\B(i,d)=\Star x_d \cup \Star y_d.
\end{equation}
where both stars are computed in $\B(i,d)$. Also, since no face of $C^*_d$ contains both $x_d$ and $y_d$, \begin{equation} \label{intersection}
\Star x_d \cap \Star y_d= \lk x_d \cap \lk y_d=\B(i-1,d-1).
\end{equation}
Here the last step is by Lemma \ref{links-intersection}, and as before the stars and links are computed in $\B(i,d)$. As stars are contractible and hence have vanishing reduced homology, an application of the Mayer-Vietoris sequence using Eqs.~(\ref{union}) and (\ref{intersection}) implies that
\begin{displaymath}
\widetilde{H}_j(\B(i,d);\Z)=\widetilde{H}_{j-1}(\B(i-1,d-1); \Z)=
\left\{\begin{array}{ll}
0, & \mbox{ if $j\neq i$}\\
\Z, & \mbox{ if $j=i$},
\end{array}
\right.
\end{displaymath}
where the last step is by inductive hypothesis. Thus $\B(i,d)$ has the same homology as $\S^i$.

To show that $\B(i,d)$ is a combinatorial manifold, recall that according to Lemma \ref{shellability}, the stars of $x_d$ and $y_d$ in $\B(i,d)$ are shellable full-dimensional proper subcomplexes of $C^*_d$, and hence combinatorial balls. By Lemma \ref{links-intersection} together with our inductive hypothesis, these two combinatorial balls intersect along a combinatorial $(d-2)$-manifold, $\B(i-1,d-1)$, that is contained in their boundaries (see Eq.~(\ref{intersection})). Therefore, the union of these balls, is a $(d-1)$-dimensional combinatorial manifold, as required. \endproof

We close this section with proofs of Lemmas \ref{links-intersection} and \ref{shellability}.

\smallskip\noindent{\it Proof of Lemma \ref{links-intersection}:\,}
Let $\sigma\in C^*_{d-1}$ and let $\tau=\Fill_{d-1}(\sigma)$. Then $\tau\cup\{x_d\}$ and $\Fill_d(\sigma\cup\{x_d\})$ have switch sets of the same cardinality, and so do $\tau\cup\{y_d\}$ and $\Fill_d(\sigma\cup\{y_d\})$. Thus we infer from Lemma \ref{small-faces} that $\sigma\in \lk_{\B(i,d)}(x_d)\cap \lk_{\B(i,d)}(y_d)$ if and only if $|\Sw_d(\tau\cup\{x_d\})|\leq i$ and $|\Sw_d(\tau\cup\{y_d\})|\leq i$.
The lemma follows since
\begin{displaymath}
\Sw_d(\tau\cup\{x_d\})\subseteq \Sw_{d-1}(\tau)\sqcup{\{d-1\}} \quad
\mbox{and} \quad
\Sw_d(\tau\cup\{y_d\})\subseteq \Sw_{d-1}(\tau)\sqcup{\{d-1\}},
\end{displaymath}
and since one of these two inclusions holds as equality.\endproof

For the proof of Lemma \ref{shellability} we need to introduce a few more definitions. We start by defining a total order, $\prec$, on the set of subsets of $[d-1]$: for $I,J\subseteq [d-1]$ define \begin{displaymath}
I\prec J \mbox{ iff } |I|<|J| \mbox{ or } (|I|=|J| \mbox{ and } I<_{\lex}J), \end{displaymath}
where $<_{\lex}$ denotes the usual lexicographic order, that is, $I<_{\lex} J$ if the minimal element in the symmetric difference of $I$ and $J$ belongs to $I$. For example, for subsets of $[3]$, we have:
\begin{displaymath}
\emptyset \prec \{1\} \prec \{2\} \prec \{3\} \prec \{1, 2\} \prec \{1, 3\} \prec \{2,3\} \prec \{1,2,3\}.
\end{displaymath}

Since $\B(i,d)$ admits a free involution that maps $x_d$ to $y_d$, to prove Lemma \ref{shellability} it is enough to show that the star of $x_d$ in $\B(i,d)$ is shellable. Recall that $\Sw$ is a map that takes as its input a facet of $C^*_d$ and outputs a subset of $[d-1]$ --- the switch set of that facet. Conversely, given a subset $J=\{j_1<j_2<\cdots<j_k\}$ of $[d-1]$, there is a {\em unique} facet of $C^*_d$ that contains $x_d$ and has $J$ as its switch set: this facet is $\Fill_d(\{z_{j_1},\ldots, z_{j_k}, x_d\})$, where $z_{j_k}=y_{j_k}$, $z_{j_{k-1}}=x_{j_{k-1}}$, and, more generally, $z_{j_{k-s}}=y_{j_{k-s}}$ for $s$ even, and $z_{j_{k-s}}=x_{j_{k-s}}$ for $s$ odd. Therefore, $\Sw$ defines a bijection between the collection of facets of $C^*_d$ containing $x_d$ and the collection of subsets of $[d-1]$, and hence also between the collection of facets of $\B(i,d)$ containing $x_d$ and between the collection of subsets of $[d-1]$ of size at most $i$. Thus the linear order $\prec$  on subsets of $[d-1]$ induces a linear order on facets of $\B(i,d)$ containing $x_d$: for such $\tau, \tau'$ we define

\begin{displaymath}
\tau \prec \tau' \mbox{ iff } \Sw(\tau) \prec \Sw(\tau').
\end{displaymath}

In addition to the switch set of a facet $\tau$ (that is merely a set of indices) it is sometimes convenient to consider the set of elements of $\tau$ that are in switch positions, that is, the set \begin{displaymath}
\SwEl(\tau):= \tau \cap(\cup_{j\in \Sw(\tau)}\{x_j,y_j\}).
\end{displaymath}

With all these definitions at our disposal, we are ready to prove Lemma \ref{shellability}. In fact, we prove the following more precise result.
\begin{lemma} \label{shelling-order}
The order $\prec$ is a shelling order of the star of $x_d$ in $\B(i,d)$: for each facet $\tau \in\Star x_d$, the restriction of $\tau$ is given by $\SwEl(\tau)$.
\end{lemma}
\begin{example} Below is the list of facets of the star of $x_4$ in $\B(2,4)$ ordered according to $\prec$ along with their switch sets and restriction sets. $$ \begin{array}{lll}
\mbox{Facet} & \mbox{Switch set}\qquad & \mbox{Restriction} \\
\{x_1,x_2,x_3,x_4\} \qquad & \emptyset & \emptyset  \\
\{y_1, x_2,x_3,x_4\} & \{1\} & \{y_1\} \\
\{y_1, y_2,x_3,x_4\} & \{2\} & \{y_2\} \\
\{y_1, y_2,y_3,x_4\} & \{3\} & \{y_3\} \\
\{x_1, y_2, x_3,x_4\} & \{1,2\} & \{x_1,y_2\}\\
\{x_1, y_2, y_3, x_4\} & \{1,3\} & \{x_1, y_3\}\\
\{x_1, x_2,y_3,x_4\} & \{2,3\} & \{x_2, y_3\}
\end{array}$$
\end{example}

\smallskip\noindent{\it Proof of Lemma \ref{shelling-order}:\,}

Consider a facet $\tau \in \Star_{\B(i,d)}(x_d)$ and a face $F \subseteq \tau$. We need to show that either there is a facet $\sigma \in \Star_{\B(i,d)}(x_d)$ such that $\sigma \prec \tau$ and $F \subseteq \sigma$ or that $F \supseteq \SwEl(\tau)$.

Suppose $F = \{z_{j_1},\ldots,z_{j_r}\}$ with $j_1 < \cdots < j_r$ and $z_{j_k} \in \{x_{j_k},y_{j_k}\}$ for all $k$, and consider the facet $\sigma := \Fill_d(F \cup x_d)$. Observe that $|\Sw(\sigma)| \leq |\Sw(\tau)|$. If $|\Sw(\sigma)| < |\Sw(\tau)|$, then $\sigma \prec \tau$, and we are done as $F\subseteq \sigma$. Hence we may further suppose that $|\Sw(\sigma)| = |\Sw(\tau)|$.

Moreover, if $j_k \in \Sw(\sigma)$, then the symbols occurring in $w(\sigma)$ in positions $j_k$ and $j_{k+1}$ are opposite to each other (one is $x$ and the other is $y$); since $F \subseteq \tau$, it then follows that there is some $\ell_k \in \Sw(\tau)$ such that $j_k \leq \ell_k < j_{k+1}$ (with the convention that $j_{r+1} = d$). Thus the $k$-th smallest entry of $\Sw(\sigma)$ is no larger than the $k$-th smallest entry of $\Sw(\tau)$, and hence $\sigma \leq_{\lex}\tau$. Therefore, either $\sigma \prec \tau$ or $\sigma = \tau$, in which case $F \supseteq \SwEl(\tau)$. \endproof

\begin{remark}
Using Lemma \ref{shellability}, it is not hard to show that the complex $\B(i,d)$ collapses (by a sequence of elementary collapses) onto $\B(i,d-1)$, which in turn collapses onto $\B(i,d-2)$, etc., until  this series of collapses reaches $\B(i,i+1)=C^*_{i+1}$. As the complex $C^*_{i+1}$ is a combinatorial $i$-dimensional sphere, results of \cite[Chapter 3]{RourkeSanderson} imply that the manifold $\B(i,d)$ is a disc bundle over $\S^i$.
\end{remark}

\section{The boundary of $\B(i,d)$}

The goal of this section is to prove that the boundary of $\B(i,d)$, $\partial\B(i,d)$, triangulates $\S^i\times \S^{d-i-2}$. Since this boundary is a $(d-2)$-dimensional subcomplex of $C^*_d$, and hence is a codimension-1 submanifold of a combinatorial sphere, the following result of Matthias Kreck \cite{Kreck} is handy.

\begin{theorem} \label{Kreck}
Let $M$ be a simply connected codimension-1 submanifold of $\S^{d-1}$, where $d\geq 6$. If $M$ has the homology of $\S^i\times\S^{d-i-2}$ and $1<i\leq\frac{d}{2}-1$, then $M$ is homeomorphic to $\S^i\times\S^{d-i-2}$.
\end{theorem}

To be able to apply Theorem \ref{Kreck}, we need a few lemmas. In the following, we denote by $\C(i,d)$ the complement of $\B(i,d)$ in $C^*_d$ (as defined in Theorem \ref{main-B}(c)).

\begin{lemma} \label{skeleton}
Let $0\leq i \leq d-1$, and let $j=\min\{i,d-i-2\}$. Then the complex $\partial\B(i,d)$ contains the entire $j$-skeleton of $C^*_d$. \end{lemma}

\begin{proof}
Consider two subcomplexes of $C^*_d$: $\B(i,d)$ and its complement $\C(i,d)$. According to Theorem \ref{main-B}(c), $\C(i,d)$ is simplicially isomorphic to $\B(d-i-2,d)$. Theorem \ref{main-B}(a), then implies that $\B(i,d)$ contains the $i$-skeleton of $C^*_d$, and $\C(i,d)$ contains the $(d-i-2)$-skeleton of $C^*_d$.  The result follows since $\partial\B(i,d)$ is the intersection of $\B(i,d)$ and $\C(i,d)$.
\end{proof}

One immediate consequence of this lemma is
\begin{corollary} \label{simply-connectedness}
For all $2\leq i \leq d-4$, the complex $\partial\B(i,d)$ is simply connected.
\end{corollary}
\begin{proof}
For $i$ in the given interval, $\min\{i,d-i-2\}\geq 2$. Hence by Lemma \ref{skeleton}, $\partial\B(i,d)$ contains the 2-skeleton of $C^*_d$, and so $\partial\B(i,d)$ is simply connected as $C^*_d$ is.
\end{proof}

We now compute homology groups of $\partial\B(i,d)$.
\begin{lemma} \label{homology} For all $1 \leq i \leq d-2$, $H_*(\partial\B(i,d);\Z) \cong H_*(\S^{i} \times \S^{d-i-2};\Z).$
\end{lemma}
\begin{proof}
By Poincar\'{e}-Lefschetz duality \cite[Theorem 3.43]{Hatcher}, $H^k(M;\Z) \cong H_{n-k}(M,\partial M;\Z)$ for any compact, orientable $n$-manifold $M$. Henceforth, we will set $M = \B(i,d)$ and assume that homology and cohomology groups are computed with coefficients in $\Z$. Moreover, since $\partial(\B(i,d))=\partial\C(i,d)$ and since $\C(i,d)$ is simplicially isomorphic to $\B(d-i-2,d)$, we assume without loss of generality that $i\leq d-i-2$.

Recall that by Theorem \ref{main-B}(d), $H_*(M)\cong H_*(\S^i)$. Since $M$ is a full-dimensional submanifold of a sphere
(namely, of $C^*_d$), it is orientable, and hence $\partial M$ is an orientable $(d-2)$-manifold without boundary. Thus $H_0(\partial M) \cong H_{d-2}(\partial M) \cong \Z$. Also since by Lemma \ref{skeleton}, $\partial M$ contains the $i$-skeleton of $C^*_d$, it follows that $H_j(\partial M)=0$ for all $0<j<i$ and $d-i-2<j<d-2$ (where the latter is by Poincar\'e duality).  In order to study all other homology groups of $\partial M$, we must examine two cases.

\smallskip\noindent\textbf{Case 1:} $i < d-i-2$.

By the Poincar\'{e}-Lefschetz duality, $H_{d-i-1}(M,\partial M) \cong H^i(M) \cong \Z$. The long exact homology sequence for the pair $(M,\partial M)$ yields
\begin{displaymath}
0 = H_{d-i-1}(M) \rightarrow H_{d-i-1}(M,\partial M) \rightarrow H_{d-i-2}(\partial M) \rightarrow H_{d-i-2}(M)=0,
\end{displaymath}
and hence $H_{d-i-2}(\partial M) \cong H_{d-i-1}(M,\partial M)\cong \mathbb{Z}$. Similarly, since $H^{d-i-1}(M) = H^{d-i-2}(M) = 0$, it follows that $H_{i}(M,\partial M) = H_{i+1}(M,\partial M) = 0$, and an analysis of (an appropriate segment of) the same long exact homology sequence shows that $H_i(\partial M) \cong H_i(M) \cong \Z$.  Also for all $i<j < d-i-2$ we have $H_{j+1}(M,\partial M) = 0$ (since $d-j-2 \neq i$), $H_{j}(M,\partial M) = 0$ (since $d-j-1 \neq i$); and, by the following exact sequence,
\begin{displaymath}
\ldots \rightarrow H_{j+1}(M,\partial M) \rightarrow H_j(\partial M) \rightarrow H_j(M) \rightarrow H_j(M,\partial M) \rightarrow \ldots, \end{displaymath}
$H_j(\partial M) \cong H_j(M)= 0$ (since $j \neq i$).

\medskip\noindent\textbf{Case 2:} $i = d-i-2$.

By Poincar\'{e}-Lefschetz duality, since $i+1 = d-i-1$, $H_{i+1}(M,\partial M) \cong \Z$ and $H_i(M,\partial M) = 0$. We examine the long exact homology sequence for the pair $(M,\partial M)$
\begin{displaymath}
\ldots \rightarrow 0 \rightarrow H_{i+1}(M,\partial M) \rightarrow H_i(\partial M) \rightarrow H_i(M) \rightarrow 0 \rightarrow \ldots
\end{displaymath}
Since $H_i(M) \cong \Z$ is a free $\Z$-module, this short exact sequence is split exact, and hence $H_i(\partial M) \cong \Z \oplus \Z$. This completes the treatment of all possible cases and establishes the claim. \end{proof}

Using the above results, the proof of Theorem \ref{main-B}(e) is almost immediate:

\smallskip\noindent{\it Proof of Theorem \ref{main-B}(e):\,}
As in the proof of Lemma \ref{homology}, we can assume without loss of generality that $i\leq \frac{d}{2}-1$. There are several cases to consider.

For $i=0$, $\B(0,d)$ is a disjoint union of two $(d-1)$-dimensional simplices, hence its boundary is a disjoint union of two $(d-2)$-spheres, and so $\partial(\B(0,d))$ triangulates $\S^0\times \S^{d-2}$.

For $i>1$, $\partial\B(i,d)$ is simply connected by Corollary \ref{simply-connectedness} and has the same homology as $\S^i\times \S^{d-i-2}$ by Lemma \ref{homology}. Theorem \ref{Kreck} then guarantees that $\partial\B(i,d)$ triangulates $\S^i\times \S^{d-i-2}$.

Finally, for $i=1$, consider the complex $\Delta$ on $3d$ vertices $\{x_1,\ldots, x_d, y_1,\ldots, y_d, t_1,\ldots, t_d\}$ generated by the facets
\begin{eqnarray*}
\{x_1,x_2, \ldots, x_d\}, \{y_1, x_2,\ldots, x_d\}, \{y_1,y_2,x_3,\ldots, x_d\}, \cdots, \{y_1,y_2, \ldots, y_d\}, \\
\{t_1, y_2,\ldots, y_d\}, \{t_1, t_2, y_3,\ldots, y_d\}, \cdots, \{t_1,t_2,\ldots, t_d\}.
\end{eqnarray*}
This complex is a shellable $(d-1)$-ball (the above order of facets is a shelling), and $\B(1,d)$ is obtained from $\Delta$ by identifying the facets $\{x_1,x_2, \ldots, x_d\}$ and $\{t_1,t_2,\ldots, t_d\}$ of this ball via the map $x_i\mapsto t_i$, $i=1,\ldots, d$. As $\B(1,d)$ is orientable, it follows that $\B(1,d)$ triangulates $\S^1\times \mathbb{B}^{d-2}$, and hence $\partial\B(1,d)$ triangulates $\S^1\times\S^{d-2}$.
\endproof

We close this section by deriving Theorem \ref{main-S} from Theorem \ref{main-B}. By Theorem \ref{main-B}(b,e), $\partial\B(i,d)$ is a cs $2d$-vertex triangulation of $\S^i\times \S^{d-i-2}$. Moreover, if $i$ is odd, then by Theorem~\ref{main-B}(b), $\B(i,d)$ admits a vertex-transitive action of the dihedral group of order $4d$.  This action induces a vertex-transitive action on $\partial\B(i,d)$. Similarly, if $d-i$ is odd, then Theorem~\ref{main-B}(b,c) implies that $\C(i,d)$ admits a vertex-transitive action of $\D_{2d}$, which in turn induces a vertex-transitive action on $\partial\C(i,d)=\partial\B(i,d)$. Otherwise, $i$ is even, and similar reasoning using Theorem \ref{main-B}(b) applies. \endproof

\section{Remarks on face numbers and Euler characteristic}
Our treatment of $\B(i,d)$ and $\partial\B(i,d)$ would be incomplete if we did not compute enumerative characteristics such as their $h$-numbers. This is done in this section. We also discuss connections to another conjecture of Sparla that concerns possible values of the Euler characteristic of cs triangulations.

One of the basic invariants of simplicial complexes are their face numbers: the $f$-vector of a $(d-1)$-dimensional simplicial
complex $\Delta$ is $f(\Delta)=(f_{-1}, f_0, \ldots, f_{d-1})$,
where $f_j$ counts the number of $j$-dimensional faces of $\Delta$. It is sometimes more convenient to work with the $h$-vector, $h(\Delta)=(h_0,h_1,\ldots, h_d)$ (or the $h$-polynomial, $h(\Delta, x):= \sum_{j=0}^{d} h_{j}x^{d-j}$) instead of the $f$-vector ($f$-polynomial, $f(\Delta, x):= \sum_{j=0}^{d} f_{j-1}x^{d-j}$, resp.). It carries the same information as the $f$-vector and is defined by the following relation:
$$
h(\Delta, x)= f(\Delta, x-1).
$$
In particular, $h_0=1$, $h_1=f_0-d$, and $h_d=(-1)^{d-1}\widetilde{\chi}(\Delta)$, where $\widetilde{\chi}(\Delta)$ denotes the reduced Euler characteristic of $\Delta$.

Following Stanley \cite{St79}, we call a $(d-1)$-dimensional simplicial complex {\em balanced} if the vertex set $V$ of $\Delta$ can be partitioned in $d$ (nonempty) sets: $V=V^1\sqcup V^2\sqcup \cdots \sqcup V^d$ (called {\em color sets}) in such a way that no two vertices from the same color set are connected by an edge. For instance, the complex $C^*_d$ (as well as all its full-dimensional subcomplexes) is balanced:  the color sets are given by $V^j=\{x_j,y_j\}$ for $1\le j \leq d$.

For a balanced complex $\Delta$, one can define the {\em flag $f$-vector} and {\em flag $h$-vector} of $\Delta$, $(f_S)_{S\subseteq [d]}$ and $(h_S)_{S\subseteq [d]}$, whose entries refine the usual $f$-and $h$-numbers, see \cite{St79}. The only properties of these numbers we will use here are that $$f_S(\Delta)=f_{|S|-1}(\Delta_S), \mbox{ where } \Delta_S:=\{\sigma\in \Delta \ : \ \sigma\subseteq \cup_{j\in S}V^j \},$$ as well as \begin{equation} \label{flag-h}
h_S(\Delta)=(-1)^{|S|-1}(\widetilde{\chi}(\Delta_S)) \quad \mbox{ and } \quad h_j(\Delta)=\sum_{S\subseteq [d], |S|=j} h_S(\Delta).
\end{equation}

As our first result we compute the $h$-vectors of complexes $\B(i,d)$.

\begin{proposition}  \label{h-numbers-B}
For all $0 \leq i \leq d-1$ and all $0 \leq j \leq d$, \begin{equation} \label{h-nums}
h_j(\B(i,d)) = \begin{cases}
\binom{d}{j} & \text{ if } j \leq i+1, \\
(-1)^{j-i-1}\binom{d}{j} & \text{ otherwise.}
\end{cases}
\end{equation}
\end{proposition}
\begin{proof}
It follows from our definition of $\B(i,d)$ that $\B(i,d)_{[d-1]}=\B(i, d-1)$.  Since $\B(i,d)$ admits a vertex-transitive action of a group (see Theorem \ref{main-B}(b)) we inductively obtain that for $S\subseteq[d]$, $\B(i,d)_{S}$ is simplicially isomorphic to $\B(i, |S|)$, where for $i\geq s$, we set $\B(i,s)=C^*_s$. By Theorem \ref{main-B}(d), we then have that
$$\widetilde{\chi}(\B(i,d)_S)= \left\{\begin{array}{ll} (-1)^{|S|-1} & \mbox{ if $|S|\leq i+1$,}\\
(-1)^i & \mbox{ otherwise,} \\
\end{array}
\right.
$$
Summing these expressions over all $S\subseteq[d]$ of size $j$ and using Eq.~(\ref{flag-h}) implies the result.
\end{proof}

From the $h$-numbers of $\B(i,d)$, we can easily compute the $h$-numbers of $\partial\B(i,d)$. To do this, we use \cite[Theorem 3.1]{Novik-Swartz} asserting that if $\Delta$ is a $(d-1)$-dimensional manifold with boundary, then for all $0 \leq j \leq d,$ \begin{equation} \label{NS3.1}
h_{d-j}(\Delta) - h_j(\Delta) = (-1)^{d-j-1}\binom{d}{j}\widetilde{\chi}(\Delta)-g_j(\partial\Delta),
\end{equation}
where $g_j(\partial\Delta) := h_j(\partial\Delta) - h_{j-1}(\partial\Delta)$, and $h_{-1}:=0$ (and so, $h_j(\partial\Delta)=\sum_{k=0}^j g_k(\partial\Delta)$).

\begin{proposition}  \label{g-numbers-boundary}
Suppose $i \leq \lfloor \frac{d-2}{2} \rfloor$. Then $$
g_k(\partial\B(i,d)) =
\left\{\begin{array}{lll}
\binom{d}{k} & \mbox{ if $k\leq i+1$,}\\
(-1)^{k-i-1}\binom{d}{k} & \mbox{ if $i+1\leq k\leq d-i-1$,}\\
-\left((-1)^{k-i}+(-1)^{d-k-i}+1\right)\binom{d}{k}  & \mbox{ if $k\geq d-i-1$.}
\end{array}
\right.
$$
\end{proposition}

\begin{proof} Substitute eq.~\eqref{h-nums} in \eqref{NS3.1} and use the fact that $\widetilde{\chi}(\B(i,d))=(-1)^i$.
\end{proof}

In addition to the $h$-numbers of simplicial complexes, one can consider the $h'$-numbers: if $\Delta$ is a $(d-1)$-dimensional simplicial complex, then for $0\leq j \leq d$,
$$
h'_j(\Delta)=h_j(\Delta)+\binom{d}{j}\sum_{k=1}^{j-1}(-1)^{j-k-1}\beta_{k-1}(\Delta), \mbox{ where $\beta_{k-1}(\Delta)=\dim_{\R}\widetilde{H}_{k-1}(\Delta; \R)$}.
$$
Thus $h'_d(\Delta)=\beta_{d-1}(\Delta)$. Furthermore, when $\Delta$ is balanced, the flag $h'$-numbers of $\Delta$ are defined and satisfy $$h'_S(\Delta)=\beta_{|S|-1}(\Delta_S)\quad \mbox{ for } S\subseteq[d].$$ These numbers refine the $h'$-numbers: $h'_j(\Delta)=\sum_{|S|=j} h'_S(\Delta)$. A proof analogous to that of Proposition \ref{h-numbers-B} yields the following.

\begin{proposition} For all $S\subseteq [d]$,
$$h'_S(\B(i,d))=\left\{\begin{array}{ll}
1 & \mbox{ if $|S|\leq i+1$,}\\
0,& \mbox{ otherwise}
\end{array} \right. $$
Hence $h'_j(\Delta)=\binom{d}{j}$ if $j\leq i+1$ and $h'_j(\Delta)=0$ if $j>i+1$.
\end{proposition}

\begin{remark}
The $h$-numbers of triangulated spheres and balls as well as the $h'$-numbers of manifolds (with and without boundary) are equal to dimensions of homogeneous components of Artinian reductions of their Stanley-Reisner rings; however this connection is beyond the scope of this paper. Using these techniques, one can show that among all $(d-1)$-dimensional triangulated manifolds with non-vanishing $\beta_i$, the complex $\B(i,d)$ has the (componentwise) minimal flag $h'$-vector. Trying to construct such a balanced complex was the starting point of this project. \end{remark}

We close the paper with a discussion of the following conjecture of Sparla on the Euler characteristic of cs triangulations of  manifolds.

\begin{conjecture} \label{sparla-conj} {\rm{(\cite[Conjecture 4.12]{Sparla-thesis}, \cite{Sparla2})}}
Let $M$ be a centrally symmetric combinatorial $2r$-dimensional manifold with $2k$ vertices. Then
\begin{equation}\label{sparla-formula}
(-1)^r\binom{2r+1}{r+1}(\chi(M)-2) \leq 4^{r+1}\binom{\frac{1}{2}(k-1)}{r+1}.
\end{equation}
Moreover, equality is attained if and only if $M$ contains the $r$-skeleton of the $k$-dimensional cross polytope. \end{conjecture}

Both assertions of this conjecture were proved in \cite{N05} under an additional restriction that $M$ has at least $6r+4$ vertices. While the first part of the conjecture remains open for $2k< 6r+4$, our construction of $\B(i,d)$ shows that the second assertion of this conjecture {\bf fails} if $2k=4r+4$ vertices. Indeed, let $M =\partial\B(i,2r+2)$. Then $M$ is a cs triangulation of $\S^i \times \S^{2r-i}$ with $2(2r+2)$ vertices, and $\chi(M) -2 = 2\cdot (-1)^i$.  When $i<r$ and $i$ has the same parity as $r$, equality holds in \eqref{sparla-formula}, but $M$ does not have  the complete $r$-skeleton of the $(2r+2)$-dimensional cross polytope since $\widetilde{H}_i(M;\mathbb{Z}) \neq 0$.

In the positive direction, it follows easily from results of \cite{N05} that Sparla's conjecture does hold for cs triangulations of manifolds all of whose Betti numbers but the middle one vanish.

\begin{proposition}
Let $\M$ be a cs triangulation of a $2r$-dimensional manifold with $2k$ vertices. If all Betti numbers of $M$ but the middle one vanish (that is, $\beta_j(M)\neq 0$ only if $j\in\{r,2r\}$), then $$\binom{2r+1}{r+1}\beta_r(M) =(-1)^r\binom{2r+1}{r+1}(\chi(M)-2) \leq 4^{r+1}\binom{\frac{1}{2}(k-1)}{r+1},
$$
and equality is attained if and only if $M$ contains the $r$-skeleton of the $k$-dimensional cross polytope. In particular, an arbitrary cs triangulation of $\S^r \times \S^r$ with $4r+4$ vertices contains the $r$-skeleton of the $(2r+2)$-dimensional cross polytope. \end{proposition}

\proof The inequality follows from \cite[Eq.~(12)]{N05}, and the treatment of
equality is the same as in \cite{N05} (see the last remark of Section 4 there). \endproof

As this paper shows, the complexes $\B(i,d)$ and $\partial\B(i,d)$ have many fascinating properties, and we hope that their further study will lead to even more new results.

\bigskip\noindent{\large \bf Acknowledgments} We are grateful to Wolfgang K\"uhnel and Felix Effenberger for helpful feedback on some preliminary results of this paper.  We also thank Joel Hass and Jack Lee for helpful conversations.

{\small
}
\end{document}